\RequirePackage{fix-cm}
\documentclass[smallextended]{svjour3}       
\smartqed  
\usepackage{graphicx}
\usepackage{amsmath}
\usepackage{amssymb}
\usepackage{cite}
\usepackage{color}

\newcommand{\dist}{{\mathop{\rm dist}}}

\newcommand{\cN}{{\cal N}}
\newcommand{\cI}{{\cal I}}
\newcommand{\cJ}{{\cal J}}
\newcommand{\cK}{{\cal K}}
\newcommand{\cF}{{\cal F}}
\newcommand{\cB}{{\cal B}}
\newcommand{\cD}{{\cal D}}

\newcommand{\cS}{{\cal S}}

\newcommand{\mN}{{\mathbb N}}

 \journalname{Mathematical Programming}
\begin{document}
\title{Necessary Optimality Conditions and Exact Penalization for Non-Lipschitz Nonlinear Programs\thanks{The first author's work was supported in part by NSFC Grant (No. 11401379) and the second author's work was supported in part by NSERC.
}
}
\subtitle{Dedicated To R. Terry Rockafellar in honor of his $80$th birthday}

\titlerunning{Non-Lipschitz Nonlinear Programs: Stationarity and Exact Penalization}         

\author{Lei Guo   \and
        Jane J. Ye 
}


\institute{Lei Guo \at
             Sino-US Global Logistics
Institute, Shanghai Jiao Tong University, Shanghai 200030, China\\
             \email{guolayne@sjtu.edu.cn} \and
             Jane Ye\at
         Department of Mathematics and
Statistics, University of Victoria, Victoria, BC, V8W 2Y2, Canada   \\
              \email{janeye@uvic.ca}
}

\date{Received: date / Accepted: date}

\maketitle

\begin{abstract}

When the objective function is not locally Lipschitz, constraint qualifications
are no longer sufficient for Karush-Kuhn-Tucker (KKT)  conditions to hold at a local minimizer, let alone ensuring an exact penalization. In this paper,
we extend  quasi-normality and  relaxed constant positive linear dependence (RCPLD) condition  to allow the non-Lipschitzness of the objective function and show that they are sufficient for KKT conditions to be necessary for optimality.
Moreover, we derive exact penalization results for the following two special cases. When the non-Lipschitz term in the objective function is the sum of a composite function of a separable lower semi-continuous function with a continuous function and an indicator function of a closed subset, we show that a local minimizer of our problem is also a local minimizer of an exact penalization problem under a local error bound condition for a restricted constraint region and a suitable assumption on the outer separable function. When the non-Lipschitz term is the sum of a continuous function and an indicator function  of a closed subset, we also show that our problem admits an exact penalization under an extended quasi-normality involving the coderivative of the continuous function.

\keywords{Non-Lipschitz program \and necessary optimality \and exact penalization \and error bound}

 \subclass{90C26 \and 90C30 \and 90C46}

\end{abstract}
\section{Introduction}
The purpose of this  paper is to study necessary optimality conditions and exact penalization for the following non-Lipschitz nonlinear program:
\begin{eqnarray}\label{eq:P0}
     \min
     && f(x)+\Phi(x)\nonumber\\
      {\rm s.t.} && g(x)\leq0,\\
                 && h(x)=0,\nonumber
\end{eqnarray}
where $f:\Re^d\to \Re,g:\Re^d\to\Re^n,h:\Re^d\to\Re^m$ are  Lipschitz around the point of interest, and $\Phi:\Re^d\to (-\infty, \infty]$ is an extended-valued lower semi-continuous function.

Including a non-Lipschitz term in the objective function has significantly enlarged the applicability of standard nonlinear programs.
For example, it has recently been discovered that when the term $\Phi$ belongs to a certain class of non-Lipschitz functions, local minimizers of problem \eqref{eq:P0} are often sparse. This property makes problem \eqref{eq:P0} useful for seeking a sparse solution in many fields such as image restoration, signal processing, wireless communication, and portfolio selection in financial applications; see, e.g., \cite{int:Donoho,Chartrand07,BianChen15,Chen2013,Liu2015}.

It is well-known that a {\it constraint qualification}  is a condition imposed on  constraint functions so that Karush-Kuhn-Tucker (KKT)   conditions hold at a local minimizer. There exist very weak constraint qualifications such as Guignard's and Abadie's constraint qualifications (\cite{Guignard,Abadie}) but they are not easy to verify since it  involves computing the tangent or normal cone of the constraint region. The challenge is to find verifiable constraint qualifications that are applicable to as many situations as possible. For  nonlinear programs where the objective function is locally Lipschitz, the verifiable classical  constraint  qualifications in the literature include linear independence constraint qualification, Slater's condition, and Mangasarian-Fromovitz constraint qualification (MFCQ). Moreover, it is well-known that when all constraint functions are linear, no constraint qualification is required for KKT conditions to hold at a local minimizer. In recent years, quite a few new and weaker verifiable constraint qualifications have been introduced; see, e.g., \cite{BertsekasJOTA2002,cpld-quasi,{gfrerer},Minchenko-stakhovski,RCPLD,abdeabi--twoCQ,ccp-2016,SCQ}. In particular, quasi-normality is a weak constraint qualification that was first introduced in \cite{BertsekasJOTA2002} and extended to locally Lipschitz programs in \cite{ye-zhang}.
The recently introduced relaxed constant positive linear dependence (RCPLD) condition  in \cite{RCPLD} is also a weak constraint qualification. Both quasi-normality and RCPLD are weaker than MFCQ and hold automatically when all constraint functions are linear (see \cite[Proposition 3.1]{BertsekasJOTA2002}). Moreover, they can admit a local error bound for the constraint region (see \cite{RCPLD,guoyezhang-infinite}) and thus by Clarke's exact penalization principle \cite[Proposition 2.4.3]{c}, they are sufficient to ensure an exact penalization when the objective function is locally Lipschitz.




Little research has been done in  KKT necessary optimality conditions for non-Lipschitz nonlinear programs in the literature, let alone exact penalization. The Fritz John type necessary optimality conditions for non-Lipschitz programs { were first given by Kruger and Mordukhovich in \cite{Kruger-Mord79} and reproved by  Mordukhovich in \cite[Theorem 1(b)]{Mor80} and  Borwein et al. in \cite[Corollary 2.6]{int:Borwein}}. For our problem, since all  functions are locally Lipschitz  except the objective function, (A1) in \cite[Corollary 2.6]{int:Borwein} never holds and consequently (A2) in \cite[Corollary 2.6]{int:Borwein} holds. Hence, the Fritz John condition \cite[Corollary 2.6]{int:Borwein} for our problem states that at a local minimizer $x^*$, there exist $\lambda\in \Re^{|\cI^*|}_+$ and $\mu\in\Re^m$ not all zero such that at least one of the following cases holds:
\begin{itemize}
\item[(i)]
$0\in  \partial^\infty \Phi (x^*) + \sum\limits_{i\in \cI^*} \lambda_i\partial g_i(x^*) + \sum\limits_{j=1}^m \partial (\mu_jh_j)(x^*)$,
\item[(ii)] $0\in   \partial(f+\Phi) (x^*)+ \sum\limits_{i\in \cI^*}  \lambda_i\partial g_i(x^*) + \sum\limits_{j=1}^m \partial (\mu_jh_j)(x^*)$,
\end{itemize}
where $\cI^*:=\{i: g_i(x^*)=0\}$, and $\partial, \partial^\infty$ denote the limiting subdifferential and the horizon subdifferential respectively (see the definitions in Section \ref{sec:pre}). Consequently, we can derive the KKT necessary optimality condition  from the above Fritz John condition immediately as follows. Suppose that there are no nonzero abnormal multipliers, i.e., the following implication holds:
\begin{eqnarray}\label{generazied mf}
\left\{\begin{array}{ll} 0\in  \partial^\infty \Phi (x^*) + \sum\limits_{i=1}^n \lambda_i\partial g_i(x^*) + \sum\limits_{j=1}^m \partial (\mu_jh_j)(x^*)\\[2pt] \lambda_i\geq0,\ \lambda_i g_i(x^*)=0 \ i=1,\ldots,n\end{array}\right. \Longrightarrow (\lambda,\mu)=0,
\end{eqnarray}
then the above condition (ii) holds, which means that $x^*$ is a KKT point. We call the implication \eqref{generazied mf}  $\partial^\infty$-no nonzero abnormal multiplier constraint qualification ($\partial^\infty$-NNAMCQ)  at $x^*$.
Note that when $\Phi$ is Lipschitz around $x^*$, we have $\partial^\infty \Phi (x^*)=\{0\}$ and hence $\partial^\infty$-NNAMCQ reduces to the standard NNAMCQ for nonlinear programs with equality and  inequality constraints. When $\Phi$ is an indicator function of a closed subset,
$\partial^\infty$-NNAMCQ reduces to the standard NNAMCQ for nonlinear programs with equality, inequality, and abstract set constraints. When $\Phi$ is neither Lipschitz around $x^*$ nor equal to an indicator function, the implication \eqref{generazied mf} involves the horizon subdifferential $\partial^\infty \Phi (x^*)$ of the non-Lipschitz term. Thus, $\partial^\infty$-NNAMCQ is no longer a constraint qualification since it is related to the objective function. However, since it is a condition under which a local minimizer is a KKT point, we call such a condition {\it a qualification condition}. Very recently, Chen et al. \cite{chenguoluye} gave some necessary optimality conditions for problem \eqref{eq:P0} where the non-Lipschitz term $\Phi$ is continuous and all the other functions are continuously differentiable  under RCPLD,  and the so-called basic qualification (BQ  for short) (see the definition in \eqref{cq2}), and  proposed an augmented Lagrangian method for solving this kind of problems. It should be noted that BQ is very difficult to verify as discussed in the paragraph after Corollary \ref{cor2}.



In this paper, we extend the standard quasi-normality and the standard RCPLD to problem \eqref{eq:P0}. Similar to  $\partial^\infty$-NNAMCQ, our new qualification conditions also involve $\partial^\infty \Phi (x^*)$ and we call them $\partial^\infty$-quasi-normality and $\partial^\infty$-RCPLD respectively. Moreover, we derive two exact penalization results for two special cases of problem \eqref{eq:P0} under some suitable  conditions.
We summarize our main contributions as follows:
\begin{itemize}
\item We introduce two new verifiable qualification conditions called  $\partial^\infty$-quasi-normality and $\partial^\infty$-RCPLD respectively and show that they are sufficient for  KKT conditions to be necessary for optimality. These two qualification conditions are both weaker than $\partial^\infty$-NNAMCQ and hold automatically when $\Phi$ is Lipschitz around the point of interest and $g,h$ are linear. As a by-product,
    we extend the standard RCPLD on smooth  constraint functions to the case where there is an extra abstract set constraint and show that KKT conditions are necessary for optimality.

\item Exact penalization for two special cases of problem \eqref{eq:P0} are derived. Case i): $\Phi$ is the sum of a composite function of a separable lower semi-continuous function with a continuous function and an indictor function of a closed subset. In this case, we show that a local minimizer of problem \eqref{eq:P0} is also a local minimizer of an exact penalization problem under a local error bound condition for a restricted constraint region and a suitable assumption on the outer separable function. Case ii): $\Phi$ is the sum of a continuous function and an indicator function of a closed subset. In this case, we introduce $D^*$-quasi-normality that is an extended quasi-normality involving  the coderivative  of the continuous function, and show that $D^*$-quasi-normality is sufficient to ensure an exact penalization. Note that $D^*$-quasi-normality reduces to the standard quasi-normality for nonlinear programs with equality, inequality, and abstract set constraints when the continuous function is Lipschitz around the point of interest.
\end{itemize}

The rest of this paper is organized as follows. In Section \ref{sec:pre} we give some background materials. In Section \ref{sec:qc} we propose some qualification conditions for problem \eqref{eq:P0}. In Section \ref{sec:nec} we derive necessary optimality conditions for problem \eqref{eq:P0} under these qualification conditions. We investigate some sufficient conditions ensuring an exact penalization for problem \eqref{eq:P0} in Section \ref{sec:exact}.

\section{Preliminaries}\label{sec:pre}


The notations used in this paper are standard in the literature. The symbol $\mN$ (resp., $\Re, \Re_+,\Re_-$)   denotes the set of nonnegative integers  (resp., real numbers, nonnegative real numbers, nonpositive real numbers).  For a finite set $T$, $|T|$ denotes its cardinality. For any $x\in \Re^d$, we denote by $x_+:=\max\{x,0\}$ the non-negative part of $x$,  $\|x\|_p:={ \large (\sum\limits_{i=1}^d|x_i|^p \large)^{1/p}}$ for any $p>0$, and $\|x\|$ any norm in $\Re^d$. Let $\cB_\delta(x)$ denote a closed ball centered at $x$ with positive radius $\delta$. The indicator function of a subset $\cD\subseteq \Re^d$ is denoted by $\delta_\cD$ and $\dist_\cD(x)$ denotes the Euclidean distance from $x$ to $\cD$. Let $\cF$ denote the constraint region for problem \eqref{eq:P0} and for any $x\in\cF$, denote by
$\cI_g(x):=\{i: g_i(x)=0\}$ the index set of active inequality constraints.

We say that $\cF$ admits a local error bound at ${\bar x}\in \cF$ if there exist $\delta>0$ and $\kappa >0$ such that $\dist_\cF(x) \le \kappa (\|h(x)\| + \|g(x)_+\|)$ for all $x\in \cB_\delta({\bar x})$.

We next give some background materials on variational analysis; see, e.g., \cite{c,clsw,Rock98,mor1} for more details. For a function $\varphi:\Re^d\to[-\infty,\infty]$ and a point $x^*\in \Re^d$ where $\varphi(x^*)$ is finite, the
regular (or Fr\'{e}chet) subdifferential of $\varphi$ at $x^*$ is defined
as
$$\hat{\partial}\varphi(x^*):=\{v: \varphi(x)\geq \varphi(x^*)+v^T(x-x^*)+o(\|x-x^*\|)\ \forall x\},$$
the limiting (or Mordukhovich) subdifferential of $\varphi$ at $x^*$ is
defined as
$$\partial\varphi(x^*):=\{v: \exists x^k\to_\varphi x^*, v^k\in \hat{\partial}\varphi(x^k) \ {\rm s.t.}\ v^k\to v\},$$
and the horizon (or singular Mordukhovich) subdifferential of $\varphi$ at
$x^*$ is defined as
$$\partial^{\infty}\varphi(x^*):=\{v: \exists x^k\to_\varphi x^*, v^k\in \hat{\partial}\varphi(x^k)\ {\rm and}\ t_k\to0\ {\rm with\ } t_k \geq0 \ {\rm s.t.}\ t_kv^k\to v\},$$
where $o(\cdot)$ means $o(\alpha)/\alpha \to 0$ as $\alpha \to 0$, and $x^k \to_\varphi x^*$ means $x^k\to x^*$ and $\varphi(x^k)\to\varphi(x^*)$ as $k\to\infty$. It is well-known that $\varphi$ is Lipschitz around $x^*$ if and only if $\partial^\infty \varphi(x^*)=\{0\}$ by \cite[Theorem 9.13]{Rock98}.

The regular (or Fr\'{e}chet) normal cone of $\cD$ at $x^*\in \cD$ is a closed convex cone defined as $\widehat{\cN}_\cD(x^*):=\hat{\partial} \delta_\cD(x^*)$ and the limiting  (or Mordukhovich) normal cone of $\cD$ at $x^*$ is
a closed cone defined as $\cN_\cD(x^*):=\partial \delta_\cD(x^*)$. We say that $\cD$ is regular at $x^*$ if $\cD$ is locally closed at $x^*$ and $\cN_\cD(x^*)=\widehat{\cN}_\cD(x^*)$.

Given a set-valued mapping $\cS:\Re^d \rightrightarrows\Re^m$ and a point $\bar{x}$ with $\cS(\bar{x})\neq\emptyset$, the coderivative of $\cS$ at $\bar{x}$ for any $\bar{u}\in \cS(\bar{x})$ is the mapping $D^*\cS(\bar{x}|\bar{u}):\Re^m \rightrightarrows \Re^d$ defined by
\[
D^*\cS(\bar{x}|\bar{u})(y):=\{v: (v,-y)\in \cN_{\rm gph \cS}(\bar{x},\bar{u})\},
\]
where ${\rm gph\, \cS}:=\{(x,y): y\in \cS(x)\}$. When $\cS$ is single-valued at $\bar{x}$ with $\cS(\bar x)=\bar u$, the notation $D^*\cS(\bar{x}|\bar{u})$ is simplified to $D^*\cS(\bar{x})$.
In the case where  $\cS$ is not only single-valued but also Lipschitz around $\bar{x}$, the coderivative is related to the limiting subdifferential by the scalarization formula:
$$D^*\cS(\bar x)(y)=\partial \langle y,\cS\rangle (\bar x) \quad \forall y \in \Re^m.$$
We say that $\cS$ is locally bounded at $\bar{x}\in \Re^d$ if there exist $M>0$ and $\delta>0$ such that
\[
\|v\| \leq M \quad \forall v\in \cS(x),\forall x\in \cB_\delta(\bar{x}).
\]
Recall from \cite[Definition 5.4]{Rock98} that $\cS$ is said to be outer semi-continuous at $\bar{x}$ if
\begin{equation*}
\{\bar{v}: \exists x^k\to \bar{x}, v^k\in \cS(x^k)\ {\rm s.t.}\ v^k\to \bar{v}\}\subseteq \cS(\bar{x}).
\end{equation*}
It is well-known that the limiting normal cone mapping, the limiting subdifferential mapping, and the horizon subdifferential mapping are all outer semi-continuous everywhere; see, e.g., \cite[Propositions 6.6 and 8.7]{Rock98}.

By using the outer semi-continuity of the limiting normal cone mapping and the definition of the coderivative, it is easy to give the following proposition that will be useful in deriving exact penalization results in Section \ref{sec:exact}.

\begin{proposition}\label{coderi conti}
The coderivative $D^*S(x|u):\Re^d\rightrightarrows\Re^m$ is outer semi-continuous in the sense that if there exists $v^k\in D^*S(x^k|u^k)(y^k)$ where $x^k\to x^*$, $y^k\to y^*$, and $u^k\to u^*$ with $u^k\in S(x^k)$ such that $v^k\to v^*$, then $v^*\in D^*S(x^*|u^*)(y^*)$.
\end{proposition}

The following proposition collects some useful properties and calculus rules of the limiting subdifferential.

\begin{proposition} \label{calculus}
\begin{itemize}
\item[\rm (i)] {\rm \cite[Exercise 10.10]{Rock98}} Let $f, g:\Re^d \to [-\infty,\infty]$ be proper lower semi-continuous around $x^* \in \Re^d$ and finite at $x^*$, and let $\alpha,\beta$ be nonnegative scalars.  Assume that at least one of them is Lipschitz around $x^*$. Then
$$ \partial (\alpha f+\beta g)(x^*) \subseteq \alpha \partial f(x^*)+\beta \partial  g(x^*).$$
Here we let $0\cdot \emptyset =\{0\}$ by convention.
\item[\rm (ii)]{\rm \cite[Theorem 2.5 and Remark (2)]{jourani1993}}  Let $g: \Re^n\to\Re^m$ be Lipschitz around $x^*$ and $f:\Re^m\to\Re $ be Lipschitz around $g(x^*)$. Then the composite function $f\circ g$ is Lipschitz  around $x^*$ and
$$\partial  (f\circ g)(x^*) \subseteq \bigcup_{\xi \in \partial  f(g(x^*))} \partial  \langle \xi, g\rangle (x^*).$$

\item[\rm (iii)]{\rm \cite[Theorem 3.38]{mor1}} Let $g: \Re^d\to\Re^m$ be continuous at $x^*$ and $f:\Re^m\to\Re $ be Lipschitz around $g(x^*)$. Then
\begin{eqnarray*}
\partial  (f\circ g)(x^*) \subseteq \bigcup_{\xi \in \partial  f(g(x^*))} D^* g(x^*) (\xi),\quad
\partial^\infty  (f\circ g)(x^*) \subseteq D^* g(x^*) (0).
\end{eqnarray*}


\item[\rm (iv)]{\rm \cite[Theorem 7.5]{Mor-nonsmooth}} Let $f({x}) := \max\{f_i({x}): i = 1, \ldots, s\}$ where $f_i : \Re^d\to \Re$ is continuous at $x^*$ for all $i=1,\dots, s$. If all but at most one of the functions $\{f_i: i=1,\ldots,s\}$ are Lipschitz around $x^*$, then
$$\partial  f(x^*) \subseteq \bigcup \left\{\sum_{i\in \cI^*} \lambda_i\diamond \partial f_i(x^*): (\lambda_1,\ldots,\lambda_s)\in\Lambda^*\right\},$$
where $\cI^*:= \{i: f_i(x^*) = f(x^*)\}$ is the index set of active indices and
$$\Lambda^*:=\left\{(\lambda_1,\ldots,\lambda_s): \lambda_i\geq0 \ i\in \cI^*,
\lambda_i=0 \ i\notin \cI^*,\ \sum_{i\in\cI^*} \lambda_i=1\right\}.$$
Here we define $\alpha\diamond \partial g$ by $\alpha\partial g$ if $\alpha>0$ and by $\partial^\infty g$ if $\alpha=0$.
 \end{itemize}
\end{proposition}

\section{Qualification conditions}\label{sec:qc}


Since the objective function of problem \eqref{eq:P0} includes a non-Lipschitz term,  KKT conditions are no longer  necessary for optimality only under constraint qualifications such as the standard quasi-normality \cite[Definition 5]{ye-zhang} and the standard RCPLD \cite[Definition 4]{RCPLD}.
In this section we extend the standard quasi-normality and the standard RCPLD to problem \eqref{eq:P0} as follows so that  KKT conditions can be necessary for optimality under the extended quasi-normality and the extended RCPLD respectively.

\begin{definition}\label{defi}
Let $x\in\cF$. (a) We say that $x$ is $\partial^\infty$-quasi-normal if there is no nonzero vector $(\lambda,\mu)\in\Re^n\times \Re^m$ such that there exists a sequence $\{x^k\}$ which converges to $x$ as $k\to\infty$ satisfying
\begin{eqnarray}
&& 0\in  \partial^\infty \Phi(x)+\sum_{i=1}^n \lambda_i\partial g_i(x) + \sum_{j=1}^m \partial (\mu_jh_j)(x),\label{eqn1}\\
&&\lambda_i\geq0,\ \lambda_ig_i(x)=0\ i=1,\ldots,n,\label{eqn2}\\
&& g_i(x^k)>0\ i\in I,\ \mu_j h_j(x^k)>0\ j\in J,\quad \forall k\in \mN,\label{eqn3}
\end{eqnarray}
where $I:=\{i:\lambda_i>0\}$, $J:=\{j:\mu_j\neq0\}$, and $\mN$ is the set of all positive integers.

(b) Assume that $g,h$ are smooth around $x$. Let $\cJ\subseteq\{1,\ldots,m\}$ be such that $\{\nabla
h_j(x): j\in \cJ\}$ is a basis for  {the} ${\rm span}\,\{\nabla h_j(x): j=1,\ldots,m\}$. We say that $\partial^\infty$-RCPLD condition holds at $x$ if there exists $\delta>0$ such that
\begin{itemize}
\item[\rm (i)] $\{\nabla h_j(y): j=1,\ldots,m\}$ has the same rank for each $y\in {\cal
B}_{\delta}(x)$;
\item[\rm (ii)] for each $\cI\subseteq \cI_g(x)$, if there exist $\{\lambda_i\geq0:
i\in\cI\}$ and $\{\mu_j: j\in\cJ\}$ not all zero such that
    \begin{equation}
 0\in  \partial^\infty \Phi(x)+\sum_{i\in\cI} \lambda_i \nabla g_i(x)+ \sum_{j\in\cJ} \mu_j \nabla h_j(x),\label{Defn2.3}
    \end{equation}
    then $\{ \nabla g_i(y), \nabla h_j(y): i\in \cI, j\in \cJ\}$ is linearly dependent for each $y\in{\cal B}_{\delta}(x)$.
\end{itemize}
\end{definition}

It is easy to see that both $\partial^\infty$-quasi-normality and $\partial^\infty$-RCPLD are weaker than
$\partial^\infty$-NNAMCQ (i.e., implication (\ref{generazied mf})) but the reverse is not true; see Examples \ref{example 1}--\ref{example 3}. Note that if $\Phi$ is Lipschitz around $x$, then $\partial^\infty \Phi(x)=\{0\}$, and thus $\partial^\infty$-quasi-normality and $\partial^\infty$-RCPLD reduce to the standard quasi-normality and the standard RCPLD respectively for nonlinear programs with equality and inequality constraints. If $\Phi$ is an indicator function of a closed subset $\Omega$, i.e., $\Phi(x)=\delta_\Omega(x)$, then $\partial^\infty \Phi(x)=\cN_\Omega(x)$ by \cite[Exercise 8.14]{Rock98}. Thus $\partial^\infty$-quasi-normality reduces to the standard quasi-normality for nonlinear programs with equality, inequality, and abstract set constraints, and
$\partial^\infty$-RCPLD allows us to extend the original definition of RCPLD \cite[Definition 4]{RCPLD} to the problem where there is an extra abstract set constraint $x\in \Omega$ since inclusion (\ref{Defn2.3}) becomes
    \[
 0\in   \sum_{i\in\cI} \lambda_i \nabla g_i(x) + \sum_{j\in\cJ} \mu_j \nabla h_j(x) +{\cal N}_\Omega(x).
    \]
In this case we simply say that RCPLD holds.

%

We next extend the standard quasi-normality to problem \eqref{eq:P0} for ensuring an exact penalization.

\begin{definition}\label{D-quasi}
Suppose that $\Phi(x):=\Psi(x)+\delta_\Omega(x)$ where $\Psi$ is a continuous function and $\Omega$ is a closed subset in $\Re^d$. Let $x\in {\cF}$.
We say that $x$ is $D^*$-quasi-normal if there is no nonzero vector $(\lambda,\mu)\in\Re^n\times \Re^m$ such that there exists a sequence $\{x^k\}$ which converges to $x$ as $k\to\infty$ satisfying \eqref{eqn2}--\eqref{eqn3} and
$$ 0\in  D^*\Psi(x)(0)+\sum_{i=1}^n\lambda_i\partial g_i(x) + \sum_{j=1}^m \partial (\mu_jh_j)(x)+\cN_\Omega(x).$$
\end{definition}

If $\Psi$ is Lipschitz  around $x$, then $D^*\Psi(x)(0)=\{0\}$ and hence $D^*$-quasi-normality reduces to the standard quasi-normality for nonlinear programs with equality, inequality, and abstract set constraints. Since $\partial^\infty \Phi(x)\subseteq D^* \Phi(x)(0)$ (see \cite[Theorem 1.80]{mor1}),  $D^*$-quasi-normality is stronger than $\partial^\infty$-quasi-normality when $\Omega=\Re^d$.

We call  problem \eqref{eq:P0} an $\ell_{1/2}$ minimization problem if  the non-Lipschitz term $\Phi(x)$ is equal to {$(\|x\|_{1/2})^{1/2}$}. The problem in the following example is an $\ell_{1/2}$ minimization problem with linear constraints. It gives an example  for which $\partial^\infty$-RCPLD, $\partial^\infty$-quasi-normality, and  $D^*$-quasi-normality are all satisfied
but $\partial^\infty$-NNAMCQ does not hold.

\begin{example}\label{example 1}\rm
Consider the following problem
\begin{eqnarray*}
\begin{array}{rl}
\min \quad & f(x):=\sqrt{|x_1|}+\sqrt{|x_2|}+ \sqrt{|x_3|}+ \sqrt{|x_4|}\\[3pt]
{\rm s.t.\quad } & g_1(x):=x_1+x_2+x_3+x_4-2\leq0,\\[2pt]
           & h_1(x):=x_1+x_2-1=0,\\[2pt]
           & h_2(x):=x_3+x_4-1=0
\end{array}
\end{eqnarray*}
at a minimizer $x^*={(1,0,1,0)}$. By direct calculation, we have $\partial^\infty f(x^*)=\{0\}\times \Re\times \{0\}\times \Re$, $\nabla g_1(x)={(1,1,1,1)}$, $\nabla h_1(x)={(1,1,0,0)}$, and $\nabla h_2(x)={(0,0, 1,1)}$ for any $x$.
Direct verification implies that there exists $(\lambda_1,\mu_1,\mu_2)\neq0$ such that
\[
0\in \partial^\infty f(x^*)+\lambda_1\nabla g_1(x^*)+\mu_1\nabla h_1(x^*)+\mu_2\nabla h_2(x^*),\ \lambda_1\geq0,
\]
which means that $\partial^\infty$-NNAMCQ does not hold at $x^*$. But in this case, the family of gradients $\{\nabla g_1(x),\nabla h_1(x),\nabla h_2(x)\}$ is linearly dependent for any $x$. Thus, $\partial^\infty$-RCPLD holds at $x^*$. To show that  $\partial^\infty$-quasi-normality also holds at $x^*$, we assume that there exist $(\lambda_1,\mu_1,\mu_2)\neq0$ and a sequence $\{x^k\}$ converging to $x^*$ such that
\begin{eqnarray}
&&0\in \partial^\infty f(x^*)+\lambda_1\nabla g_1(x^*) +\mu_1\nabla h_1(x^*)+\mu_2\nabla h_2(x^*),\ \lambda_1\geq0,\label{example1}\\
&& g_1(x^k)>0\ {\rm if}\ \lambda_1>0,\ \mu_1h_1(x^k)>0\ {\rm if}\ \mu_1\neq0,\  \mu_2h_2(x^k)>0\ {\rm if}\ \mu_2\neq0.\label{example2}
\end{eqnarray}
By \eqref{example1}, it follows that $\mu_1=\mu_2=-\lambda_1$. Thus $\lambda_1>0$ and $\mu_1=\mu_2<0$. These together with \eqref{example2} leads to
\[
2<x_1^k+x_2^k+x_3^k+x_4^k<2.
\]
This contradiction shows that there is no nonzero $(\lambda_1,\mu_1,\mu_2)$ satisfying \eqref{example1}--\eqref{example2}. Thus, $\partial^\infty$-quasi-normality holds at $x^*$. Since it is easy to verify that $D^*f(x^*)(0)=\partial^\infty f(x^*)$, $D^*$-quasi-normality holds at $x^*$ as well. \qquad $\square$
\end{example}

The following example of an $\ell_{1/2}$ minimization problem with nonlinear constraints illustrates that it is possible that $\partial^\infty$-quasi-normality holds but $\partial^\infty$-RCPLD does not hold.

\begin{example}\rm\label{example 3}
Consider the following problem
\begin{eqnarray*}
\begin{array}{rl}
\min \quad & f(x):= \sqrt{|x_1|}+\sqrt{|x_2|}+ \sqrt{|x_3|}\\[3pt]
{\rm s.t.}\quad & g(x):=x_1+x_2-x_3^2-1\leq0,\\[2pt]
           & h(x):=x_1+x_2-1=0
\end{array}
\end{eqnarray*}
at a minimizer $x^*={(1,0,0)}$.  $\partial^\infty$-RCPLD does not hold at $x^*$ since there exists $(\lambda,\mu)\neq0$ such that
\begin{eqnarray*}
0\in \partial^\infty f(x^*)+\lambda\nabla g(x^*)+\mu\nabla h(x^*),\ \lambda\geq0
\end{eqnarray*}
but $\{\nabla g(x),\nabla h(x)\}$ is linearly independent for any $x$ with $x_3\not =0$. To show that  $\partial^\infty$-quasi-normality holds at $x^*$, assume that there exist $(\lambda,\mu)\neq0$ and a sequence $\{x^k\}$ converging to $x^*$ such that
\begin{eqnarray}
&& 0\in \partial^\infty f(x^*)+\lambda\nabla g(x^*)+\mu\nabla h(x^*),\ \lambda\geq0,\label{example3}\\
&& g(x^k)>0\ {\rm if}\ \lambda>0,\ \mu h(x^k)>0\ {\rm if}\ \mu\neq0.\label{example4}
\end{eqnarray}
By \eqref{example3}, it follows that $\lambda+\mu=0$. Thus $\lambda>0$ and $\mu<0$. But by \eqref{example4}, we have $1+(x_3^k)^2<1$. The contradiction shows that  $\partial^\infty$-quasi-normality holds at $x^*$. \qquad $\square$
\end{example}

We know that the standard RCPLD does not imply the standard quasi-normality (see \cite[Example 2]{RCPLD}). It is also easy to see that the standard RCPLD and the standard quasi-normality at a local minimizer $x^*$ of the problem
\begin{equation}\label{example}
\min\ f(x) \quad {\rm s.t.}\ x\in {\cal F}
\end{equation}
are equivalent to $\partial^\infty$-RCPLD and $\partial^\infty$-quasi-normality at a local minimizer $(x^*,0)$ of the perturbed problem
\[
\min\ f(x)+\sqrt{|z|} \quad {\rm s.t.}\ x\in {\cal F}, z\in \Re,
\]
respectively. Thus, if problem \eqref{example} is such that the standard RCPLD holds but the standard  quasi-normality does not hold, then   $\partial^\infty$-RCPLD holds but $\partial^\infty$-quasi-normality does not hold for the above perturbed problem.


We next give some characterizations for  $\partial^\infty$-quasi-normality,  $\partial^\infty$-RCPLD, and  $D^*$-quasi-normality in terms of the standard quasi-normality and the standard RCPLD.

\begin{proposition}\label{prop suff} Let $x^*\in\cF$.
{\rm (i)} If $\partial^\infty$-quasi-normality holds at $x^*$, then both the standard quasi-normality and
the following basic qualification (BQ) hold at $x^*$:
\begin{equation}\label{cq2}
-\partial^\infty \Phi(x^*)\cap \cN_\cF(x^*) =\{0\}.
\end{equation}
If $g,h$ are smooth around $x^*$, then  $\partial^\infty$-quasi-normality at $x^*$ is equivalent to the standard quasi-normality plus BQ \eqref{cq2} at $x^*$.

{\rm (ii)} $\partial^\infty$-RCPLD holds at $x^*$ if and only if both the standard RCPLD and BQ \eqref{cq2} hold at $x^*$.

{\rm (iii)} If $D^*$-quasi-normality holds at $x^*$, then the standard quasi-normality holds at $x^*$ and
\begin{equation}\label{cq1}
-D^* \Psi(x^*)(0)\cap \cN_\cF(x^*) =\{0\}.
\end{equation}
If $\Omega$ is regular and $g,h$ are smooth around $x^*$, then $D^*$-quasi-normality holds at $x^*$ if and only if  both the standard quasi-normality  and condition (\ref{cq1}) hold at $x^*$.
\end{proposition}
\begin{proof}
(i) Suppose that $\partial^\infty$-quasi-normality holds at $x^*$. Then it is easy to see that the standard quasi-normality also holds at $x^*$ since $0\in \partial^\infty \Phi(x^*)$. Thus by \cite[Proposition 4]{ye-zhang}, it follows that
\begin{equation}\label{enhm}
\cN_\cF(x^*)\subseteq \left\{\begin{array}{l}\sum\limits_{i\in\cI^*} \lambda_i\partial g_i(x^*)+ \sum\limits_{j=1}^m \partial (\mu_jh_j)(x^*):  \lambda_i\geq0\ i\in\cI^*,\ \exists \{x^k\}\to x^*\\ {\rm s.t.}\ g_i(x^k)>0\ i\in I, \mu_j h_j(x^k)>0\ j\in J,\quad \forall k\in \mN\end{array}\right\},
\end{equation}
where
\begin{equation}\label{index}
\cI^*:=\cI_g(x^*),\ I:=\{i:\lambda_i>0\},\ J:=\{j:\mu_j\neq0\}.
\end{equation}
We now show that BQ (\ref{cq2}) holds. By contradiction, suppose that $0\neq\zeta\in -\partial^\infty \Phi(x^*)\cap \cN_\cF(x^*)$. Then by \eqref{enhm}, there exist $(\lambda,\mu)\in \Re^{|\cI^*|}\times\Re^m$ and a sequence $\{x^k\}$ converging to $x^*$ such that
\begin{eqnarray*}
&&\zeta\in\sum_{i\in\cI^*} \lambda_i\partial g_i(x^*) + \sum_{j=1}^m \partial (\mu_jh_j)(x^*),\\
&& \lambda_i\ge0 \ i\in\cI^*,\ g_i(x^k)>0\ i\in I,\ \mu_j h_j(x^k)>0\ j\in J,\quad \forall k\in\mN.
\end{eqnarray*}
Since $0 \neq\zeta\in-\partial^\infty \Phi(x^*)$, it then follows that $(\lambda,\mu)\neq0$ and
\begin{eqnarray}
&& 0\in  \partial^\infty \Phi(x^*)+\sum_{i\in\cI^*} \lambda_i\partial g_i(x^*) + \sum_{j=1}^m \partial (\mu_jh_j)(x^*),\label{nec1}\\
&&\lambda_i\geq0\ i\in\cI^*,\ g_i(x^k)>0\ i\in I,\ \mu_j h_j(x^k)>0\ j\in J,\quad \forall k\in\mN, \label{nec2}
\end{eqnarray}
which contradicts $\partial^\infty$-quasi-normality. Thus, BQ \eqref{cq2} holds.

We next show the converse part. Assume that $g,h$ are smooth around $x^*$, and both the standard quasi-normality and BQ \eqref{cq2} hold at $x^*$.  By contradiction, assume that $\partial^\infty$-quasi-normality does not hold at $x^*$. That is, there exist $0\neq(\lambda,\mu)\in\Re^{|\cI^*|}\times \Re^m$ and a sequence $\{x^k\}$ converging to $x^*$ such that
\begin{eqnarray}
&& 0\in  \partial^\infty \Phi(x^*)+\sum_{i\in\cI^*} \lambda_i\nabla g_i(x^*) + \sum_{j=1}^m \mu_j\nabla h_j(x^*),\label{nec1new}\\
&&\lambda_i\geq0\ i\in\cI^*,\ g_i(x^k)>0\ i\in I,\ \mu_j h_j(x^k)>0\ j\in J, \quad \forall k\in\mN, \label{nec2new}
\end{eqnarray}
where $\cI^*,I,J$ are defined as in \eqref{index}.
Moreover, since $g,h$ are smooth around $x^*$, it follows from \cite[Theorem 6.14]{Rock98} that
\begin{equation}
\left\{\sum_{i\in \cI^*} \lambda_i \nabla g_i(x^*)+\sum_{j=1}^m \mu_j\nabla h_j(x^*):  \lambda_i\geq 0\ i\in \cI^*, \mu\in \Re^m\right\}\subseteq \cN_\cF(x^*).\label{smoothN}
\end{equation}
This and \eqref{nec1new} imply that
\[
\sum_{i\in\cI^*} \lambda_i\nabla g_i(x^*) + \sum_{j=1}^m \mu_j\nabla h_j(x^*)\in -\partial^\infty \Phi(x^*)\cap \cN_\cF(x^*),
\]
which together with BQ \eqref{cq2} means that
\[
\sum_{i\in\cI^*} \lambda_i\nabla g_i(x^*) + \sum_{j=1}^m \mu_j\nabla h_j(x^*)=0.
\]
This together with \eqref{nec2new} and the relation $(\lambda,\mu)\neq0$ contradicts the standard quasi-normality. Thus, $\partial^\infty$-quasi-normality holds at $x^*$.

(ii) Let the standard RCPLD and BQ \eqref{cq2} hold at $x^*$. Let $\cJ$ be the index set given in the definition of the standard RCPLD such that $\{\nabla h_j(x^*): j\in \cJ\}$ is a basis for the  ${\rm span}\,\{\nabla h_j(x^*): j=1,\ldots,m\}$, and let $\cI^*$ be defined as in \eqref{index}. To show $\partial^\infty$-RCPLD, it suffices to show that Definition \ref{defi}(b)(ii) holds.
Assume that there exist nonzero vectors $\{\alpha_i\geq0: i\in \cI\}$ with $\cI\subseteq \cI^*$ and $\{\beta_j: j\in \cJ\}$
such that
\[
0\in \partial^\infty\Phi(x^*)  + \sum_{i\in\cI} \alpha_i\nabla g_i(x^*) + \sum_{j\in \cJ} \beta_j\nabla h_j(x^*),
\]
which together with \eqref{cq2} and \eqref{smoothN} implies that
\[
\sum_{i\in \cI} \alpha_i \nabla g_i(x^*) + \sum_{j\in\cJ} \beta_j\nabla h_j(x^*)=0.
\]
This and the standard RCPLD imply the existence of $\delta>0$ such that
$$\{\nabla g_i(x),\nabla h_j(x): i\in \cI, j\in \cJ\}\ {\rm is\ linearly\ dependent\ for\ each}\ x\in{\cal B}_{\delta}(x^*).$$
Thus, Definition \ref{defi}(b)(ii) holds and then $\partial^\infty$-RCPLD holds at $x^*$.

To show the converse part, suppose that $\partial^\infty$-RCPLD holds at $x^*$.  It then follows immediately that RCPLD holds at $x^*$ since $0\in \partial^\infty \Phi(x^*)$. Then by \cite[Theorem 3.2]{Guo14}, it follows that
 \begin{equation}\label{respre}
\cN_\cF(x^*)\subseteq \left\{ \sum_{i\in \cI^*} \lambda_i \nabla g_i(x^*) + \sum_{j=1}^m\mu_j\nabla h_j(x^*): \lambda_i\geq0\ i\in \cI^*, \mu\in \Re^m\right\}.
\end{equation}
We now show that BQ \eqref{cq2} holds. To the contrary, assume  that $0\neq\zeta\in -\partial^\infty \Phi(x^*)\cap \cN_\cF(x^*)$. Let $\cJ$ be such that $\{\nabla h_j(x^*): j\in \cJ\}$ is a basis for the ${\rm span}\,\{\nabla h_j(x^*): j=1,\ldots,m\}$. Then by \eqref{respre}, there exist $\{\lambda_i>0: i\in \cI\}$ with $\cI\subseteq\cI^*$  and $\{\mu_j:j\in\cJ\}$ not all zero such that
\begin{equation}\label{contra1}
\zeta=\sum_{i\in \cI} \lambda_i \nabla g_i(x^*) + \sum_{j\in \cJ}\mu_j\nabla h_j(x^*),
\end{equation}
which together with the relation $\zeta\in -\partial^\infty \Phi(x^*)$ implies that
\begin{equation*}
0\in \partial^\infty \Phi(x^*) +\sum_{i\in \cI} \lambda_i \nabla g_i(x^*) + \sum_{j\in \cJ}\mu_j\nabla h_j(x^*).
\end{equation*}
Then by Definition \ref{defi}(b)(ii), there exist $\{\alpha_i:i\in\cI\}$ and $\{\beta_j: j\in \cJ\}$ not all zero such that
\begin{equation*}
\sum_{i\in \cI} \alpha_i \nabla g_i(x^*)+\sum_{i\in \cJ}\beta_j\nabla h_j(x^*)=0,
\end{equation*}
which together with \eqref{contra1} implies that for any $\gamma \in \Re$,
\begin{equation*}
\zeta=\sum_{i\in \cI} (\lambda_i-\gamma\alpha_i) \nabla g_i(x^*) + \sum_{j\in \cJ}(\mu_j-\gamma\beta_j)\nabla h_j(x^*).
\end{equation*}
Choosing $\gamma\neq0$ as the smallest number such that $\lambda_i-\gamma\alpha_i=0$ for at least one $i\in \cI$, we are able to represent $\zeta$ with at least one fewer vectors $\nabla g_i(x^*)$. We may repeat this procedure until $\zeta=\sum_{j\in \cJ}\theta_j\nabla h_j(x^*)$ for some $\{\theta_j: j\in \cJ\}$ not all zero. Then by the relation $\zeta\in -\partial^\infty \Phi(x^*)$, it follows that
\begin{equation*}
0\in \partial^\infty \Phi(x^*) +\sum_{j\in \cJ}\theta_j\nabla h_j(x^*).
\end{equation*}
Thus by Definition \ref{defi}(b)(ii), $\{\nabla h_j(x^*): j\in \cJ\}$  must be linearly dependent. This contradicts the fact that $\{\nabla h_j(x^*): j\in \cJ\}$ is a basis. Hence BQ \eqref{cq2} holds.

(iii) When $g,h$ are smooth around $x^*$ and $\Omega$ is regular, it follows from \cite[Theorem 6.14]{Rock98} that
\begin{equation}
\left\{
\begin{array}{c}\sum\limits_{i\in \cI^*} \lambda_i \nabla g_i(x^*)+\sum\limits_{j=1}^m \mu_j\nabla h_j(x^*)+{\cal N}_\Omega(x^*): \\[2pt] \lambda_i\geq 0\ i\in \cI^*, \mu\in \Re^m
\end{array}
\right\}
\subseteq \cN_\cF(x^*). \label{smoothNnew}
\end{equation}
The proof for (iii) is exactly the same as that for (i) except that $\partial^\infty \Phi(x^*)$ and (\ref{smoothN}) are replaced by $D^*\Psi(x^*)(0)$ and (\ref{smoothNnew}), respectively. \qquad $\square$
\end{proof}

When the constraint region is so simple that its limiting normal cone is easy to calculate directly, we can use Proposition \ref{prop suff} to verify our proposed qualification conditions. The following simple minimax problem illustrates that conditions \eqref{cq2}--\eqref{cq1} hold and then $\partial^\infty$-quasi-normality, $\partial^\infty$-RCPLD, and $D^*$-quasi-normality are all satisfied since the constraint function is linear.

\begin{example} \rm
Consider the following minimax problem
\begin{eqnarray*}
\min\limits_{x\leq0}\max\limits_{y\geq0}\ -x^3+xy.
\end{eqnarray*}
Let $V(x):=\max\{-x^3+xy: y\geq0\}$. Then it is easy to verify that the above minimax problem can be equivalently rewritten as
\begin{equation}\label{valueP}
\min¡¡\ V(x)\quad {\rm s.t.} \ x\leq0,
\end{equation}
where $V(x)= -x^3$ if $ x\leq0$ and $\infty$ otherwise.
Clearly, $x^*=0$ is  a  minimizer of problem \eqref{valueP}. We observe that $V$ is not continuous at $x^*$ but is lower semi-continuous at $x^*$.  Moreover,
\[
\partial^\infty V(x^*)=D^* V(x^*)(0)=\cN_{\Re_-}(x^*)=\Re_+,
\]
which indicates that $-\partial^\infty V(x^*)\cap \cN_{\Re_-}(x^*)=-D^* V(x^*)(0)\cap \cN_{\Re_-}(x^*)=\{0\}$.  Since the constraint function of problem \eqref{valueP} is linear, the standard RCPLD obviously holds and by \cite[Proposition 3.1]{BertsekasJOTA2002}, the standard quasi-normality is also satisfied. It then follows from Proposition \ref{prop suff} that  $\partial^\infty$-quasi-normality, $\partial^\infty$-RCPLD, and $D^*$-quasi-normality  are all satisfied at $x^*$ for problem \eqref{valueP}. \qquad $\square$ 
\end{example}

By using the outer semi-continuity of the horizon subdifferential mapping and Proposition \ref{coderi conti}, it is not difficult to show that both $\partial^\infty$-quasi-normality and  $D^*$-quasi-normality are locally persistent as follows.

\begin{proposition}
If $\partial^\infty$-quasi-normality ($D^*$-quasi-normality) holds at $x^*\in \cF$, then there exists $\delta_0>0$ such that $\partial^\infty$-quasi-normality ($D^*$-quasi-normality) holds at every point in $\cB_{\delta_0}(x^*)\cap \cF$.
\end{proposition}


We now show that $\partial^\infty$-RCPLD is also locally persistent.

\begin{proposition}\label{prop per}
If $\partial^\infty$-RCPLD
holds at $x^*\in\cF$, then there exists $\delta_0>0$ such that $\partial^\infty$-RCPLD
holds at every point in $ \cB_{\delta_0}(x^*)\cap\cF$.
\end{proposition}
\begin{proof}
Assume that $\partial^\infty$-RCPLD holds at $x^*$. Let $\cJ\subseteq\{1,\ldots,m\}$
be such that $\{\nabla h_j(x^*): j\in \cJ\}$ is a basis for {the} ${\rm
span}\,\{\nabla h_j(x^*): j=1,\ldots,m\}$. Then it is easy to see that there exists $\delta_1\in(0,\delta)$ such that $\{\nabla h_i(x): i\in \cI\}$ is linearly independent for all $x\in\cB_\delta(x^*)$, where $\delta$ is given in Definition \ref{defi}(b). Then it follows from Definition \ref{defi}(b)(i) that $\{\nabla h_j(x): j\in \cJ\}$ is a basis for {the} ${\rm
span}\,\{\nabla h_j(x): j=1,\ldots,m\}$ for any $x\in \cB_{\delta_1}(x^*)$. Let $\delta_2:=\delta_1/2$. Then by Definition \ref{defi}(b)(i) again, $\{\nabla h_j(y): j=1,\ldots,m\}$ has the same rank for all $y\in\cB_{\delta_2}(x)$ and  $x\in\cB_{\delta_2}(x^*)$. Thus it suffices to show that there exists $\delta_0\in(0,\delta_2)$ such that for any $x\in\cB_{\delta_0}(x^*)$, Definition \ref{defi}(b)(ii) holds at $x$. Assume to the contrary that this is not true. That is, there exist a sequence $\{x^k\}$ converging to $x^*$, and $\{\lambda^k_i\geq 0: i\in\cI^k\}$ with $\cI^k\subseteq\cI_g(x^k)$ and $\{\mu^k_j:j\in\cJ\}$ not all zero  such that
\begin{eqnarray} \label{prop per1}
0\in \partial^\infty\Phi(x^k) + \sum_{i\in\cI^k} \lambda^k_i\nabla g_i(x^k) + \sum_{j\in \cJ} \mu^k_j\nabla h_j(x^k),
\end{eqnarray}
and there exists a sequence $\{y^{k,l}\}_l$ converging to $x^k$ such that
\begin{eqnarray*}
\{\nabla g_i(y^{k,l}),\nabla h_j(y^{k,l}): i\in\cI^k, j\in\cJ\}{\rm \ is\ linearly\ independent\ for\ all}\ l.
\end{eqnarray*}
By the diagonalization law, there exists a sequence $\{z^k\}$ converging to $x^*$ such that
\begin{eqnarray}\label{gra li}
\{\nabla g_i(z^k), \nabla h_j(z^k): i\in\cI^k, j\in\cJ\}\ {\rm is\ linearly\ independent\ for\ all}\ k.
\end{eqnarray}
Since $g$ is continuous, it is easy to verify that $\cI_g(x^k)\subseteq \cI_g(x^*)$ for any $k$ sufficiently large and hence $\cI^k\subseteq\cI_g(x^*)$. Since the number of the possible sets $\cI^k$ is finite, without loss of generality we may assume that $\cI^k\equiv \cI$ for any $k$ sufficiently large. Let
\[
t_k:=\max\{\lambda^k_i, |\mu^k_j|: i\in \cI,j\in\cJ\}.
\]
Clearly, $t_k>0$ for any $k$. Without loss of generality, we may assume that
\begin{eqnarray*}
\frac{\lambda^k_i}{t_k} \to \lambda^*_i\geq0\quad i\in \cI,\quad \frac{\mu^k_j}{t_k} \to \mu^*_j\quad j\in\cJ\quad {\rm as}\ k\to\infty.
\end{eqnarray*}
It is easy to see that
\[
\max\{\lambda^*_i,|\mu^*_j|: i\in \cI,j\in\cJ\}=1.
\]
By the outer semi-continuity of the horizon subdifferential mapping, dividing \eqref{prop per1} by $t_k$ and taking limits on both sides as $k\to\infty$ imply that
\[
0\in \partial^\infty\Phi(x^*)  + \sum_{i\in\cI} \lambda^*_i\nabla g_i(x^*) + \sum_{j\in \cJ} \mu^*_j\nabla h_j(x^*).
\]
The last two relations and $\partial^\infty$-RCPLD imply that for any $x\in\cB_\delta(x^*)$,
\begin{eqnarray*}
\{\nabla g_i(x),\nabla h_j(x): i\in\cI, j\in\cJ\} \ {\rm is\ linearly\ dependent,}
\end{eqnarray*}
which contradicts \eqref{gra li}. The desired result follows immediately. \qquad  $\square$
\end{proof}

\section{Necessary optimality conditions}\label{sec:nec}

The purpose of this section is to show that the  KKT condition defined below is necessary for optimality under $\partial^\infty$-quasi-normality or $\partial^\infty$-RCPLD.

\begin{definition}[KKT condition] \label{defi-kkt}\rm
Let $x^*\in\cF$. We say that $x^*$ is a KKT point of problem \eqref{eq:P0} if there exist multipliers $\lambda\in \Re^n$ and $\mu\in\Re^m$ such that
\begin{eqnarray*}
&&0\in \partial f(x^*) + \partial \Phi(x^*) + \sum_{i=1}^n \lambda_i\partial g_i(x^*) + \sum_{j=1}^m \partial (\mu_jh_j)(x^*),
\\
&& \lambda_i\geq0,\ \lambda_i g_i(x^*)=0 \ i=1,\ldots,n.
\end{eqnarray*}
\end{definition}



%

We first show that the KKT condition holds at a local minimizer under a weaker qualification condition.

\begin{lemma}\label{thm nopti}
Let $x^*$ be a local minimizer of problem \eqref{eq:P0}.  Suppose that BQ \eqref{cq2} holds at $x^*$ and
\begin{equation}\label{normalcone}
\cN_\cF(x^*)\subseteq \left\{\sum_{i\in \cI^*} \lambda_i\partial g_i(x^*) + \sum_{j=1}^m \partial (\mu_jh_j)(x^*): \lambda_i\geq0\ i\in \cI^*, \mu\in \Re^m\right\},
\end{equation}
where $\cI^*:=\cI_g(x^*)$. Then $x^*$ is a KKT point.
\end{lemma}
\begin{proof}
It is clear that $x^*$ is a local minimizer of the problem
\begin{equation*}
\min \quad f(x)+\Phi(x)+\delta_\cF(x).
\end{equation*}
Then by Fermat's rule (see, e.g., \cite[Theorem 10.1]{Rock98}), we have
\[
0\in \partial f(x^*)+ \partial(\Phi+\delta_\cF)(x^*).
\]
Since BQ \eqref{cq2} holds at $x^*$, it then follows from the sum rule for  the limiting subdifferentials (see, e.g., \cite[Corollary 10.9]{Rock98}) and the relation $\partial\delta_\cF(x^*)=\cN_\cF(x^*)$ that
\begin{equation*}
0\in \partial f(x^*)+ \partial\Phi(x^*)+\cN_\cF(x^*).
\end{equation*}
This and \eqref{normalcone} imply the desired result immediately. \qquad $\square$
\end{proof}

The following result follows immediately from the fact that \eqref{normalcone} may be implied by the local error bound condition (e.g., \cite[Proposition 3.4]{int:loffe}).

\begin{corollary}\label{cor2}
Let $x^*$ be a local minimizer of problem \eqref{eq:P0}. If BQ \eqref{cq2} holds at $x^*$ and $\cF$  admits a local error bound at $x^*$, then $x^*$ is a KKT point.
\end{corollary}

Let us revisit Example \ref{example 1}  which is in a four dimensional space. Even in this low dimensional space, it is not easy to calculate the limiting normal cone of the constraint region and hence   BQ \eqref{cq2} is difficult to verify. For a constraint region involving many nonlinear constraints in high-dimensional spaces, it is almost impossible to calculate directly the limiting normal cone and thus BQ \eqref{cq2} is very difficult to verify. Fortunately, 
$\partial^\infty$-quasi-normality and $\partial^\infty$-RCPLD  are expressed in terms of the problem data explicitly and hence much easier to verify. The following result shows that these two proposed qualification conditions are  sufficient for the KKT condition to hold at a local minimizer.

\begin{theorem}\label{thm-opt}
Let $x^*$ be a local minimizer of problem \eqref{eq:P0}.  Assume that either $\partial^\infty$-quasi-normality holds at $x^*$ or  $\partial^\infty$-RCPLD holds at $x^*$ and $g,h$ are smooth around $x^*$. Then $x^*$ is a KKT point.
\end{theorem}
\begin{proof}
By Proposition \ref{prop suff}, either the standard quasi-normality and BQ \eqref{cq2} or the standard RCPLD and BQ \eqref{cq2} hold. It then follows from \cite[Proposition 4]{ye-zhang} and \cite[Theorem 3.2]{Guo14} that condition \eqref{normalcone} holds. Thus, the desired result follows from Lemma \ref{thm nopti} immediately.  \qquad $\square$
\end{proof}


\begin{corollary}\label{cor4}
Let $x^*$ be a local minimizer of problem \eqref{eq:P0} and let  $g,h$ be linear. Suppose also that the following implication holds:
\begin{eqnarray}
&& 0\in  \partial^\infty \Phi(x^*)+\sum_{i\in\cI^*} \lambda_i\nabla g_i(x^*) + \sum_{j=1}^m \mu_j \nabla h_j(x^*),\ \lambda_i\geq0\ i\in\cI^*,\nonumber\\ [4pt]
&&\Longrightarrow \sum_{i\in\cI^*} \lambda_i\nabla g_i(x^*) + \sum_{j=1}^m \mu_j \nabla h_j(x^*)=0, \quad {\rm where}\ \cI^*:=\cI_g(x^*).\label{implic}
\end{eqnarray}
Then $x^*$ is a KKT point.
\end{corollary}
\begin{proof}
We first show that $\partial^\infty$-RCPLD holds at $x^*$. Since $h$ is linear, Definition \ref{defi}(b)(i) holds. It then suffices to show that  Definition \ref{defi}(b)(ii) holds. Let $\cJ$ be such that $\{\nabla h_j(x^*):j\in\cJ\}$ is a basis for  {the} ${\rm span}\,\{\nabla h_j(x^*):j=1,\ldots,m\}$ and $\cI\subseteq \cI^*$. Assume that there exist $\{\lambda_i:i\in \cI\}$ and $\{\mu_j:j\in\cJ\}$ not all zero such that
\[
0\in \partial^\infty \Phi(x^*)+\sum_{i\in\cI} \lambda_i\nabla g_i(x^*) + \sum_{j\in\cJ} \mu_j \nabla h_j(x^*),\ \lambda_i\geq0\ i\in\cI,
\]
which implies that $\sum\limits_{i\in\cI} \lambda_i\nabla g_i(x^*) + \sum\limits_{j\in\cJ} \mu_j \nabla h_j(x^*)=0$ by \eqref{implic}.  This means that the family of gradients $\{\nabla g_i(x), \nabla h_j(x):i\in \cI,j\in \cJ\}$ is linearly dependent for all $x$
since $g,h$ are linear. Thus, $\partial^\infty$-RCPLD holds at $x^*$ and then the desired result follows immediately from Theorem \ref{thm-opt}.  \qquad $\square$
\end{proof}

The following example illustrates the applicability of Corollary \ref{cor4}.

\begin{example}\rm
Consider the following problem
\begin{eqnarray*}
\begin{array}{rl}
\min \quad & f(x):= \sqrt{|x_1|}+\sqrt{|x_2|}\\[3pt]
{\rm s.t.}\quad & g_1(x):=x_1+x_2-1\geq 0,\\
& g_2(x):=x_1+x_2-1\le 0
\end{array}
\end{eqnarray*}
at a minimizer $x^*={(1,0)}$. By simple calculation, we have $\partial^\infty f(x^*)=\{0\}\times \Re$ and then
\[
0\in \partial^\infty f(x^*)-\lambda_1 \nabla g_1(x^*) +\lambda_2 \nabla g_2(x^*)
\]
implies that $\lambda_1=\lambda_2$. Thus {$-\lambda_1 \left(\begin{array}{c}1\\1\end{array}\right) +\lambda_2\left(\begin{array}{c}1\\1\end{array}\right) =\left(\begin{array}{c}0\\0\end{array}\right)$} and by Corollary \ref{cor4}, it then follows that $x^*$ is a KKT point. \qquad $\square$
\end{example}

Letting $\Phi$ be an indicator function of a closed subset $\Omega$ in $\Re^d$, i.e., $\Phi(x)=\delta_\Omega(x)$, the following result follows immediately from Theorem \ref{thm nopti}, which extends the result of \cite[Corollary 1]{RCPLD} to allow an extra abstract set constraint $x\in \Omega$ and the nonsmoothness of the objective function.

\begin{corollary}\label{cor3.1}
Let $x^*$ be a local minimizer of the nonlinear program
\begin{eqnarray*}
     \min_{x\in \Omega} && f(x) \nonumber\\
     {\rm s.t.}   && g(x)\leq0,\\
                  && h(x)=0,\nonumber
\end{eqnarray*}
where $f,g,h$ are defined as in problem \eqref{eq:P0} and $\Omega$ is a closed subset in $\Re^d$. Here we assume that $g,h$ are smooth around $x^*$. Suppose further that RCPLD holds at $x^*$, i.e., there exists $\delta>0$ such
that
\begin{itemize}
\item[\rm (i)] $\{\nabla h_j(x): j=1,\ldots,m\}$ has the same rank for each $x\in {\cal
B}_{\delta}(x^*)$;
\item[\rm (ii)] for each $\cI\subseteq \cI_g(x^*)$, if there exist $\{\lambda_i\geq0: i\in\cI\}$ and $\{\mu_j: j\in\cJ\}$ not all zero such that
    \begin{equation*}
 0\in  \sum_{i\in\cI} \lambda_i \nabla g_i(x^*) + \sum_{j\in\cJ} \mu_j \nabla h_j(x^*)+\cN_\Omega(x^*),
    \end{equation*}
    then $\{\nabla g_i(x), \nabla h_j(x): i\in \cI, j\in \cJ\}$ is
linearly dependent for each $x\in{\cal B}_{\delta}(x^*)$,
\end{itemize}
where $\cJ\subseteq\{1,\ldots,m\}$ is such that $\{\nabla h_j(x): j\in \cJ\}$ is a basis for {the} ${\rm span}\,\{\nabla h_j(x): j=1,\ldots,m\}$.
Then there exist multipliers $\lambda\in \Re^n$ and $\mu\in\Re^m$ such that
\begin{eqnarray*}
&&0\in \partial f(x^*) + \sum_{i=1}^n \lambda_i\nabla g_i(x^*) + \sum_{j=1}^m \mu_j\nabla h_j(x^*)+\cN_\Omega(x^*),
\\
&& \lambda_i\geq0,\ \lambda_i g_i(x^*)=0 \ i=1,\ldots,n.
\end{eqnarray*}
\end{corollary}

\section{Exact penalization} \label{sec:exact}

This section focuses on exact penalization for problem \eqref{eq:P0}. We first give an exact penalization result for a special case of problem \eqref{eq:P0} where $\Phi$ is the sum of a composite function of a separable lower semi-continuous function with a continuous function and an indictor function of a closed subset. To this end, we give a characterization of the regular subdifferential as follows. It can be shown easily by using the definition of the regular subdifferential and thus we omit the proof here.

\begin{lemma}\label{prop equi}
Let $\psi:\Re^d\to (-\infty, \infty]$ be lower semi-continuous and $x^*\in\Re^d$ be such that $\psi(x^*)$ is finite. Then $\hat{\partial}\psi(x^*)=\Re^d$ if and only if for any $M>0$, there exists $\delta>0$ such that
\[
\psi(x)-\psi(x^*) \geq M \|x-x^*\|\quad \forall x\in \cB_\delta(x^*).
\]
\end{lemma}

We are now ready to give the first main result on exact penalization.

\begin{theorem}\label{thm exact}
Assume that $x^*$ is a local minimizer of problem \eqref{eq:P0} where $$\Phi(x):=\sum_{i=1}^s \phi_i(\omega_i(x))+\delta_\Omega(x).$$ Here $\Omega$ is a closed subset in $\Re^d$ and for any $i=1,\ldots,s$, $\phi_i:\Re\to\Re$ is lower semi-continuous and $\omega_i(x):\Re^d\to\Re$ is continuous.  Let $t^*:=\omega(x^*)$, $\cI:=\{i: \partial^\infty \phi_i(t_i^*)=\{0\}\}$, and $\cI^c$ be the complement of $\cI$ with respect to $\{1,\ldots,s\}$. Assume further that $\hat{\partial} \phi_i(t^*_i)=\Re$ for any $i\in \cI^c$ and the following restricted system with respect to $(x,t)$:
\begin{equation*}
\left\{\begin{array}{l}
g(x)\leq0,\ h(x)=0,\ x\in \Omega,\\
w_i(x)-t_i=0\ i=1,\ldots,s,\\
t_i-t_i^*=0 \ i\in\cI^c
\end{array}\right.
\end{equation*}
admits a local error bound at $(x^*,t^*)$.  Then there exists $\rho_0>0$ such that for any $\rho\geq\rho_0$, $x^*$ is also a local minimizer of the exact penalization problem
\begin{eqnarray*}
\min_{x\in\Omega} \quad f(x)+\sum_{i=1}^s \phi_i(\omega_i(x)) + \rho \left(\|g(x)_+\| + \|h(x)\|\right).
\end{eqnarray*}
\end{theorem}
\begin{proof}
Since $x^*$ is a local minimizer of problem \eqref{eq:P0}, it is not difficult to see that $(x^*,t^*)$ is a local minimizer of the following problem:
\begin{eqnarray}\label{auxi p}
\begin{array}{rl}
\min\limits_{x\in \Omega} & \Pi(x,t):=f(x)+\sum\limits_{i\in\cI} \phi_i(t_i)\\[3pt]
{\rm s.t.} & g(x)\leq0,\ h(x)=0,\\[2pt]
           & w_i(x)-t_i=0\quad i=1,\ldots,s,\\[2pt]
           & t_i-t_i^*=0\quad i\in \cI^c.
\end{array}
\end{eqnarray}
We observe that $\Pi$ is Lipschitz around $(x^*,t^*)$ and denote by $L_\Pi$ the Lipschitz constant. Then by Clarke's exact penalization principle \cite[Proposition 2.4.3]{c}, there exists $\delta_1>0$ such that
\begin{eqnarray}\label{lip}
\Pi(x^*,t^*) \leq  \Pi(x,t) + L_\Pi {\rm dist}_{\cF'}(x,t)\quad \forall (x,t)\in \cB_{\delta_1}(x^*,t^*)\cap (\Omega\times \Re^s),
\end{eqnarray}
where $\cF'$ denotes the constraint region of problem \eqref{auxi p}. Since $\cF'$ admits a local error bound at $(x^*,t^*)$, there exist $\delta_2\in(0,\delta_1)$ and $\kappa>0$ such that  for any $(x,t)\in \cB_{\delta_2}(x^*,t^*)\cap (\Omega\times \Re^s)$,
\begin{eqnarray*}\label{err}
{\rm dist}_{\cF'}(x,t) \leq \kappa \left(\sum_{i=1}^s |w_i(x)-t_i| + \sum_{i\in\cI^c}|t_i-t_i^*| + \|g(x)_+\|_1 + \|h(x)\|_1\right).
\end{eqnarray*}
{For simplicity, the above local error bound is expressed under the $\ell_1$ norm.} This together with  \eqref{lip} implies that for any $(x,t)\in \cB_{\delta_2}(x^*,t^*)\cap (\Omega\times \Re^s)$,
\begin{equation}\label{au pen}
\begin{array}{l}
\Pi(x^*,t^*) \leq \Pi(x,t) + \\[2pt]
\qquad\qquad\ \kappa L_\Pi \left(\sum\limits_{i=1}^s |w_i(x)-t_i| + \sum\limits_{i\in\cI^c}|t_i-t_i^*| + \|g(x)_+\|_1 + \|h(x)\|_1\right).
\end{array}
\end{equation}
Due to the continuity of function $w$, we may choose $\delta_3\in(0,\delta_2)$ such that $(x,w(x))\in \cB_{\delta_2}(x^*,t^*)$ for any $x\in \cB_{\delta_3}(x^*)$. Thus, by letting $t=w(x)$ in \eqref{au pen}, it follows that for any $x\in\cB_{\delta_3}(x^*)\cap \Omega$,
\begin{equation}\label{au pen1}
\Pi(x^*,t^*) \leq \Pi(x,w(x)) + \kappa L_\Pi \left(\sum_{i\in\cI^c}|w_i(x)-t_i^*| + \|g(x)_+\|_1 + \|h(x)\|_1\right).
\end{equation}
Since $\hat{\partial} \phi_i (t_i^*)=\Re$ for any $i\in \cI^c$, it then follows from Lemma \ref{prop equi} and the continuity of $w$ that there exists $\delta\in(0,\delta_3)$ such that
\begin{equation*}
\phi_i(w_i(x))-\phi_i(t_i^*)\geq \kappa L_\Pi|w_i(x)-t_i^*|\quad \forall x\in\cB_{\delta}(x^*)\ \forall i\in\cI^c.
\end{equation*}
This and \eqref{au pen1} imply that for any $x\in\cB_{\delta}(x^*)\cap \Omega$,
\begin{eqnarray*}
f(x^*)+\sum_{i=1}^s\phi_i(t_i^*) &=& \Pi(x^*,t^*)+ \sum_{i\in\cI^c}\phi_i(t_i^*)\\
                 &\leq& \Pi(x,w(x)) + \sum_{i\in\cI^c}\phi_i(w_i(x))+ \sum_{i\in\cI^c}\phi_i(t_i^*)-\sum_{i\in\cI^c}\phi_i(w_i(x))\\
                 &&+\kappa L_\Pi\left(\sum_{i\in\cI^c}|w_i(x)-t_i^*| + \|g(x)_+\|_1 + \|h(x)\|_1\right)\\
                 &\leq& f(x)+\sum_{i=1}^s\phi_i(w_i(x)) -\kappa L_\Pi\sum_{i\in\cI^c}|w_i(x)-t_i^*|\\
                 && + \kappa L_\Pi\left(\sum_{i\in\cI^c}|w_i(x)-t_i^*|+ \|g(x)_+\|_1+\|h(x)\|_1\right)\\
                 &=& f(x)+\sum_{i=1}^s\phi_i(w_i(x))+ \kappa L_\Pi\left( \|g(x)_+\|_1+\|h(x)\|_1\right).
\end{eqnarray*}
Then the desired result follows immediately by the equivalence of all norms in finite dimensional spaces. \qquad $\square$
\end{proof}

It should be noted that Theorem \ref{thm exact} can be applied to a class of sparse optimization problems. For the widely used bridge penalty $\phi(t)=|t|^p$ with $p\in (0,1)$ in the sparse optimization literature, it is easy to see that $\phi$ is not Lipschitz around $t^*=0$. However, it is not hard to verify that $\hat{\partial}\phi(t^*)=\Re$ and thus this bridge penalty function is a suitable outer function required in Theorem \ref{thm exact}. In the following, we give some exact penalization results for problem \eqref{eq:P0} where the objective function is related to the bridge penalty function.

The following result shows that the problem considered in \cite{Chen13} with an extra abstract constraint set which is the union of finitely many polyhedral sets admits an exact penalization.

\begin{corollary}\label{cor1}
Assume that $x^*$ is a local minimizer of problem \eqref{eq:P0} where $$\Phi(x):=\sum_{i=1}^s |a_i^Tx|^p+\delta_\Omega(x)$$ with $a_i\in\Re^d$, $p\in (0,1)$, and  $\Omega\subseteq \Re^d$ which is the union of finitely many polyhedral sets. Assume further that $g,h$ are linear. Then there exists $\rho_0>0$ such that for any $\rho\geq \rho_0$, $x^*$ is also a local minimizer of the exact penalization problem
\begin{eqnarray*}
\min_{x\in\Omega} \quad f(x)+ \sum_{i=1}^s |a_i^Tx|^p +\rho \left(\|g(x)_+\|+\|h(x)\|\right).
\end{eqnarray*}
\end{corollary}
\begin{proof}
Let $\phi(t):=|t|^p$ and $t_i^*:=a_i^T x^*\ i=1,\ldots,s$. It is easy to verify that
\[\cI:=\{i: \partial^\infty \phi(t_i^*)=\{0\}\}=\{i: t_i^*\neq0\}\]
and $\hat{\partial}\phi(t_i^*)=\Re$ for any $i\in \cI^c$ where $\cI^c$ is the complement of $\cI$ with respect to $\{1,\ldots,s\}$. {Since the constraint set
\[
\left\{
(x,t,p):\begin{array}{l} g(x)+p^g\leq0,\ h(x)+p^h=0,\ t_i-t_i^*+p^t_i=0\  i\in \cI^c\\[2pt] a_i^T x-t_i+p^a_i=0\ i=1,\ldots,s,\ x\in \Omega  \end{array}
\right\}
\]
is the union of finitely many polyhedral sets, it then follows from the corollary in \cite[Page 210]{Robinson81} that the local error bound condition holds everywhere for the constraint set
\[
\left\{
(x,t)\in \Omega\times \Re^s:\begin{array}{l} g(x)\leq0,\ h(x)=0,\ t_i-t_i^*=0\  i\in \cI^c\\[2pt] a_i^T x-t_i=0\ i=1,\ldots,s\end{array}
\right\}.
\]
}
Then the desired result follows immediately from Theorem \ref{thm exact}.  \qquad $\square$
\end{proof}

We next give an exact penalization result for a problem which is more general than the one considered in \cite{YLiu}.

\begin{corollary}
Assume that $x^*$ is a local minimizer of problem \eqref{eq:P0} where $$\Phi(x):=\sum_{i=1}^s [(b_i-a_i^Tx)_+]^p+\delta_\Omega(x)$$ with $a_i\in\Re^d,b_i\in\Re$, $p\in(0,1)$, and $\Omega\subseteq \Re^d$ which is the union of finitely many polyhedral sets.  Assume further that $g,h$ are linear. Then there exists $\rho_0>0$ such that for any $\rho\geq \rho_0$, $x^*$ is also a local minimizer of the exact penalization problem
\begin{eqnarray*}
\min_{x\in\Omega} \quad f(x)+ \sum_{i=1}^s [(b_i-a_i^Tx)_+]^p +\rho \left(\|g(x)_+\|+\|h(x)\|\right).
\end{eqnarray*}
\end{corollary}
\begin{proof}
Let $\phi(t):=|t|^p$ and $t_i^*:=(b_i-a_i^Tx^*)_+\ i=1,\ldots,s$. Using the same notations $\cI,\cI^c$ as in the proof of Corollary \ref{cor1}, it suffices to investigate the local error bound condition for the constraint set
\[
\left\{(x,t)\in\Omega\times \Re^s: \begin{array}{l}g(x)\leq0,\ h(x)=0,\ \ t_i-t_i^*=0\  i\in \cI^c\\[2pt] (b_i-a_i^Tx)_+-t_i=0\  i=1,\ldots,s\end{array}\right\}.
\]
It is easy to see that the parametric counterpart of the above set
\[
\left\{(x,t,p): \begin{array}{l}g(x)+p^g\leq0,\ h(x)+p^h=0,\ \ t_i-t_i^*+p_i^t=0\  i\in \cI^c\\[2pt] (b_i-a_i^Tx)_+-t_i+p^{+}_i=0\  i=1,\ldots,s, \ x\in \Omega\end{array}\right\}
\]
is the  union of finitely many polyhedral sets. Thus by the corollary in \cite[Page 210]{Robinson81}, the desired local error bound condition is satisfied. The proof is complete by applying Theorem \ref{thm exact}. \qquad $\square$
\end{proof}

In the rest of this section, we investigate sufficient conditions ensuring an exact penalization for problem \eqref{eq:P0} where $\Phi$ is the sum of a continuous function and an indictor function of a closed subset. In particular,  we investigate exact penalization for the following problem:
\begin{eqnarray}\label{newp}
\min_{x\in \Omega} && f(x)+\Psi(x)\nonumber\\
{\rm s.t.}  && g(x)\leq0,\\
           &&  h(x)=0,\nonumber
\end{eqnarray}
where $f,g,h$ are defined as in problem \eqref{eq:P0}, $\Psi:\Re^d\to\Re$ is a continuous function, and $\Omega$ is a closed subset in $\Re^d$. As discussed in Section 1, when the objective function of a nonlinear program is locally Lipschitz, the admittance of the local error bound for its constraint region is sufficient to ensure an exact penalization. For this purpose, we introduce the following auxiliary problem where the objective function is locally Lipschitz:
\begin{eqnarray}\label{auxi pnew}
\min_{(x,y)\in \Omega\times\Re}        && f(x)+y\nonumber\\
{\rm s.t.} && \Psi(x)-y=0,\\
           && g(x)\leq0, \ h(x)=0.\nonumber
\end{eqnarray}
In the case where $\Omega=\Re^d$, the constraint region of problem  (\ref{auxi pnew}) can be rewritten as
\[
\Lambda=\{(x,y)\in {\rm gph}\,\Psi: g(x)\leq0,\ h(x)=0\}.
\]
We observe that by using the definition of the coderivative, the inclusion
$$ 0\in D^*\Psi(x)(0)+\sum_{i\in\cI_g(x)} \lambda_i\partial g_i(x) + \sum_{j=1}^m \partial (\mu_jh_j)(x)$$
can be rewritten as
$$(0,0)\in \sum_{i\in\cI_g(x)}\left(\begin{array}{c}\lambda_i\partial g_i(x)\\ 0\end{array}\right)+ \sum_{j=1}^m \left(\begin{array}{c}\partial (\mu_jh_j)(x)\\ 0 \end{array}\right) +\cN_{{\rm gph}\Psi} (x, \Psi(x)).$$
Hence if $x^*$ is $D^*$-quasinormal for problem (\ref{newp}), then $(x^*,\Psi(x^*))$ is quasinormal for problem (\ref{auxi pnew}). Moreover, ${\rm gph}\,\Psi$ is a closed subset in $\Re^{d+1}$ by the continuity of $\Psi$. These and the local Lipschitzness of $g,h$ enable us to apply \cite[Corollary 5.3]{guoyezhang-infinite} to derive the local error bound condition at $(x^*,\Psi(x^*))$, that is, there exist $\delta>0$ and $\kappa>0$ such that
 \[
{\rm dist}_\Lambda(x,y)\leq\kappa \left(\|g(x)_+\|+\|h(x)\|\right)\quad \forall (x,y)\in \cB_\delta(x^*,\Psi(x^*))\cap {\rm gph}\,\Psi.
\]
The exact penalization result then follows from applying Clarke's exact penalization principle \cite[Proposition 2.4.3]{c} to problem \eqref{auxi pnew}.
Unfortunately, the above argument does not work for the case where the abstract constraint set $\Omega$ is not equal to the whole space  $\Re^d$. Nevertheless, we have succeeded in deriving the following local error bound result under $D^*$-quasi-normality given in Definition \ref{D-quasi}.

\begin{lemma}\label{lema nerror}
Suppose that $D^*$-quasi-normality holds at $x^*\in \cF$. Then  the set
\begin{equation}\label{Lambda}
\Lambda:=\{(x,y)\in \Omega\times \Re : \Psi(x)-y=0,\ g(x)\leq0,\ h(x)=0\}
\end{equation}
admits a local error bound at $(x^*,y^*)$  with $y^*:=\Psi(x^*)$, that is, there exist $\delta>0$ and $\kappa>0$ such that
\begin{equation*}
{\rm dist}_\Lambda(x,y) \leq \kappa \left(|\Psi(x)-y|+ \|g(x)_+\| + \|h(x)\|\right) \  \forall (x,y)\in \cB_\delta(x^*,y^*)\cap (\Omega \times \Re).
\end{equation*}
\end{lemma}
\begin{proof}
First we observe that $\Lambda$ defined in \eqref{Lambda} can be rewritten as
\[
\Lambda =\{(x,y):  \Xi(x,y)+{\rm dist}_\Omega(x)=0\},
\]
where
$$
\Xi(x,y):=\max\left\{H(x,y),g_1(x), \ldots, g_n(x),|h_1(x)|,\ldots, |h_m(x)|\right\}
$$
with $H(x,y):=|G(x,y)|$ and $G(x,y):=\Psi(x)-y$. Then in order to obtain the desired result, by \cite[Theorem 3.1]{error02}, it suffices to show that there exist $\tilde{\delta}>0$ and $\tilde{\kappa}>0$ such that
\begin{eqnarray}\label{thm error contry}
\|\pi\|\geq \tilde{\kappa} \quad  \forall \pi \in \partial (\Xi+{\rm dist}_\Omega)(x,y),\forall  (x,y)\in\cB_{\tilde{\delta}}(x^*,y^*)\cap (\Omega \times \Re) \backslash \Lambda.
\end{eqnarray}

We now make some preparations for subsequent analysis.  Since $|\cdot|$ is globally Lipschitz and $G$ is continuous, it follows from Proposition \ref{calculus}(iii) that for any $(x,y)$,
\begin{eqnarray}
\partial H(x,y) &\subseteq& \bigcup_{\xi\in \partial|G(x,y)|} D^*G(x,y)(\xi)
                \subseteq \bigcup_{\xi\in \partial|G(x,y)|}\left(\begin{array}{c} D^*\Psi(x)(\xi)\\-\xi\end{array}\right),\label{thm error abs1}\\
\partial^\infty H(x,y) &\subseteq& D^*G(x,y)(0) \subseteq \left(\begin{array}{c} D^*\Psi(x)(0)\\0\end{array}\right).\label{thm error abs2}
\end{eqnarray}
Since ${\rm dist}_\Omega(\cdot)$ is globally Lipschitz and $\Xi$ is continuous, by Proposition \ref{calculus}(i),   we have that for any $(x,y)$,
\begin{eqnarray}\label{thm error dist}
\partial (\Xi+{\rm dist}_\Omega)(x,y) \subseteq \partial\Xi(x,y) + \partial {\rm dist}_\Omega(x)\times \{0\}.
\end{eqnarray}
Since $g,h$ are both Lipschitz around $x^*$, by Proposition \ref{calculus}(iv), it follows that for any $(x,y)$ with $x$ sufficiently close to $x^*$, there exists $(\alpha,\beta,\gamma)\in{\cal M}(x,y)$ where
\begin{eqnarray*}
{\cal M}(x,y):= \left\{(\alpha,\beta,\gamma): \begin{array}{l}\alpha\geq0,\ \beta\geq0,\ \gamma\geq0,\ \alpha+\|\beta\|_1+\|\gamma\|_1=1\\[3pt] \alpha (H(x,y)-\Xi(x,y))=0\\[3pt] \beta_i(g_i(x)-\Xi(x,y))=0\quad i=1,\ldots,n \\[3pt] \gamma_j(|h_j(x)|-\Xi(x,y))=0\quad j=1,\ldots,m\end{array}\right\}
\end{eqnarray*}
such that
\begin{equation}\label{thm error max}
\begin{array}{l}\partial\Xi(x,y) \subseteq \\
\bigcup\limits_{(\alpha,\beta,\gamma)\in {\cal M}(x,y)}\left\{ \alpha \diamond \partial H(x,y) + \sum\limits_{i=1}^n \left(\begin{array}{c}\beta_i \partial g_i(x)\\0\end{array}\right) + \sum\limits_{j=1}^m \left(\begin{array}{c} \gamma_j \partial |h_j|(x)\\0\end{array}\right)\right\}.
\end{array}
\end{equation}

In the following, we prove \eqref{thm error contry} by contradiction. Assume to the contrary that there exist a sequence $\{(x^k,y^k)\}$ with $(x^k,y^k)\in \Omega\times\Re\backslash \Lambda$ converging to $(x^*,y^*)$  and $\pi^k\in \partial (\Xi+\dist_\Omega)(x^k,y^k)$ such that $\pi^k\to 0$.
Then it follows from \eqref{thm error dist}--\eqref{thm error max} that there exists $(\alpha^k,\beta^k,\gamma^k)\in{\cal M}(x^k,y^k)$ such that
\begin{equation}\label{thm error inclu}
\pi^k\in \alpha^k \diamond\partial H(x^k,y^k) + \sum_{i=1}^n \left(\begin{array}{c}\beta^k_i\partial g_i(x^k)\\0\end{array}\right) + \sum_{j=1}^m \left(\begin{array}{c}  \gamma^k_j\partial |h_j|(x^k)\\0\end{array}\right)+\left(\begin{array}{c} \partial {\rm dist}_\Omega(x^k)\\0\end{array}\right).
\end{equation}
Noting that $(x^k,y^k)\in \Omega\times\Re\backslash \Lambda$, we have that
\begin{equation}\label{thm infes}
\Xi(x^k,y^k)>0\quad \forall k.
\end{equation}
Since $(\alpha^k,\beta^k,\gamma^k)\in {\cal M}(x^k,y^k)$, it follows that $\alpha^k\geq0$, $\beta^k\geq0,\ \gamma^k\geq0$, and
\begin{eqnarray}
&& \alpha^k + \|\beta^k\|_1 + \|\gamma^k\|_1=1,\label{thm error normal}\\
&& \alpha^k(H(x^k,y^k)-\Xi(x^k,y^k))=0,\label{thm error cp1}\\
&&\beta^k_i(g_i(x^k)-\Xi(x^k,y^k))=0\ i=1,\ldots,n,\label{thm error cp2}\\
&&\gamma_j^k(|h_j(x^k)|-\Xi(x^k,y^k))=0\ j=1,\ldots,m.\label{thm error cp3}
\end{eqnarray}
Define
\[
\bar{\gamma}^k_j:={\rm sign}(h_j(x^k))\gamma^k_j,\quad {\rm where}\ {\rm sign}(0):=0
\]
Since it follows from \eqref{thm infes} and \eqref{thm error cp3} that $\gamma_j^k=0$ when $h_j(x^k)=0$, it is easy to see that $\|\gamma^k\|_1=\|\bar{\gamma}^k\|_1$. It then follows from \eqref{thm error normal} that
\begin{equation}\label{thm nerror normal}
\alpha^k + \|\beta^k\|_1 + \|\bar{\gamma}^k\|_1=1.
\end{equation}
Moreover, by Proposition \ref{calculus}(ii), we have that
\begin{equation}\label{thm nerror gamma}
\gamma^k_j\partial |h_j|(x^k) = \partial (\bar{\gamma}^k_jh_j)(x^k).
\end{equation}
We continue the proof by considering the two separate cases as follows.

Case (a): There exists a subsequence $\{\alpha^k\}_{k\in{\cal K}}$ with ${\cal K}\subseteq \mN$ such that $\alpha^k=0$ for any $k\in {\cal K}$. Then it follows from \eqref{thm error abs2}, \eqref{thm error inclu}, \eqref{thm nerror gamma}, and the definition of notation $\diamond$ that for any $k\in {\cal K}$,
\begin{equation}\label{thm nerror 1}
\pi^k  \in \left(\begin{array}{c} D^*\Psi(x^k)(0)\\0\end{array}\right)+ \sum_{i=1}^n \left(\begin{array}{c}\beta^k_i\partial g_i(x^k)\\0\end{array}\right) + \sum_{j=1}^m \left(\begin{array}{c}  \partial (\bar{\gamma}^k_jh_j)(x^k)\\0\end{array}\right)+\left(\begin{array}{c} \partial {\rm dist}_\Omega(x^k)\\0\end{array}\right).
\end{equation}
In this case, it follows from \eqref{thm nerror normal} that $\|\beta^k\|_1+\|\bar{\gamma}^k\|_1=1$. Thus, there must exist subsequences $\{\beta^k\}_{k\in \cK_1}$ and $\{\bar{\gamma}^k\}_{k\in \cK_1}$ with $\cK_1\subseteq \cK$ such that as $\cK_1 \ni k\to \infty$,
\begin{equation}\label{thm error mul}
\beta^k \to \beta^*\geq0,\ \bar{\gamma}^k\to\gamma^*\quad {\rm with}\quad \|\beta^*\|_1+ \|\gamma^*\|_1=1.
\end{equation}
Taking limits on both sides of \eqref{thm nerror 1}, it then follows from \eqref{thm error mul}, Proposition \ref{coderi conti}, and the local boundedness of the limiting subdifferential of local Lipschitz functions that
\begin{equation}\label{thm nerror 2}
0 \in D^*\Psi(x^*)(0) + \sum_{i=1}^n \beta^*_i\partial g_i(x^*) + \sum_{j=1}^m  \partial (\gamma^*_jh_j)(x^*) + \partial {\rm dist}_\Omega(x^*).
\end{equation}
If $g_i(x^*)<0$, then $g_i(x^k)<0$ for any $k$ sufficiently large. Thus by \eqref{thm infes} and \eqref{thm error cp2}, $\beta_i^k=0$ for any $k$ sufficiently large. This together with \eqref{thm error mul} implies that $\beta^*_i=0$. In conclusion, we have
\begin{eqnarray}\label{thm nerror 3}
\beta^*_i\geq0,\ \beta^*_ig_i(x^*)=0\quad i=1,\ldots,n.
\end{eqnarray}
Moreover, if $\beta_i^*>0$, then by \eqref{thm error mul}, we have $\beta^k_i>0$ for any $k\in \cK_1$ sufficiently large. This and \eqref{thm error cp2} imply that $g_i(x^k)=\Xi(x^k,y^k)>0$. If $\gamma_j^*\neq0$, then by \eqref{thm error mul}, we have  $\gamma_j^*\bar{\gamma}^k_j>0$ for any $k\in \cK_1$ sufficiently large. Thus, it follows from the definition of  $\bar{\gamma}^k_j$ and the relation $\gamma^k_j\geq0$ that $\gamma_j^*h_j(x^k)>0$ for any $k\in \cK_1$ sufficiently large. Thus, we have
\[
\beta^*_i>0 \Longrightarrow g_i(x^k)>0,\quad \gamma^*_j\neq0 \Longrightarrow \gamma^*_j h_j(x^k)>0,
\]
which together with \eqref{thm error mul}--\eqref{thm nerror 3} contradicts $D^*$-quasi-normality at $x^*$ by using the relation $\partial {\rm dist}_\Omega(x^*)\subseteq \cN_\Omega(x^*)$.

Case (b): There exists a subsequence $\{\alpha^k\}_{k\in{\cal K}}$ with ${\cal K}\subseteq{\mN}$ such that $\alpha^k>0$ for any $k\in {\cal K}$. In this case, it then follows from \eqref{thm error abs1}, \eqref{thm error inclu}, \eqref{thm nerror gamma}, and the definition of notation $\diamond$ that for any $k\in {\cal K}$, there exists $\xi^k\in \partial|G(x^k,y^k)|$ such that
\begin{equation}\label{thm nerror 4}
\begin{array}{l}
\pi^k  \in \left(\begin{array}{c} D^*\Psi(x^k)(\alpha^k\xi^k)\\[2pt]-\alpha^k\xi^k\end{array}\right) + \sum\limits_{i=1}^n \left(\begin{array}{c}\beta^k_i\partial g_i(x^k)\\[2pt]0\end{array}\right) \\[2pt]
\qquad \qquad\qquad\qquad\quad\qquad\ + \sum\limits_{j=1}^m \left(\begin{array}{c}  \partial (\bar{\gamma}^k_jh_j)(x^k)\\[2pt]0\end{array}\right)+\left(\begin{array}{c}\partial {\rm dist}_\Omega(x^k)\\ 0\end{array}\right).
\end{array}
\end{equation}
Since $\alpha^k>0$, it follows from \eqref{thm infes} and \eqref{thm error cp1} that $H(x^k,y^k)=\Xi(x^k,y^k)>0$. Thus, by direct calculation, we have that
\begin{eqnarray}
\left\{\begin{array}{ll}\xi^k=1 &{\rm if}\ \Psi(x^k)-y^k>0,\\[2pt] \xi^k=-1 &{\rm otherwise.} \end{array}\right.
\end{eqnarray}
It then follows that $|\bar{\alpha}^k|=\alpha^k$ where $\bar{\alpha}^k:=\alpha^k\xi^k$. Thus by \eqref{thm nerror normal}, it follows that
\begin{equation*}
|\bar{\alpha}^k| + \|\beta^k\|_1 + \|\bar{\gamma}^k\|_1=1.
\end{equation*}
Without loss of generality, we assume that as ${\cal K}\ni k\to\infty$,
\begin{equation}\label{thm nerror 5}
\bar{\alpha}^k\to \alpha^*, \ \beta^k \to \beta^*,\ \bar{\gamma}^k\to\gamma^* \quad {\rm with}\quad |\alpha^*| + \|\beta^*\|_1 + \|\gamma^*\|_1=1.
\end{equation}
It then follows from \eqref{thm nerror 4} and the relation $\pi^k\to 0$ that $\bar{\alpha}^k\to0$ as ${\cal K}\ni k\to\infty$.  Thus by \eqref{thm nerror 5}, we have that
\[
\alpha^*=0,\ \|\beta^*\|_1 + \|\gamma^*\|_1=1.
\]
Taking limits on both sides of \eqref{thm nerror 4}, it then follows from \eqref{thm nerror 5}, Proposition \ref{coderi conti}, and the local boundedness of the limiting subdifferential of local Lipschitz functions that
\begin{equation*}
0\in D^*\Psi(x^*)(0) +\sum_{i=1}^n \beta_i^* \partial g_i(x^*) +\sum_{j=1}^m \partial (\gamma_j^* h_j)(x^*) + \partial {\rm dist}_\Omega(x^*).
\end{equation*}
The rest of the proof for case (b) is similar to that for case (a).

Therefore, there exist $\tilde{\delta}>0$ and $\tilde{\kappa}>0$ such that \eqref{thm error contry} holds and thus by \cite[Theorem 3.1]{error02}, we obtain the desired result immediately.  \qquad $\square$
\end{proof}

We are now ready to give the exact penalization result for problem (\ref{newp}).

\begin{theorem}\label{1thm exact}
Let $x^*$ be a local minimizer of problem (\ref{newp}).  If $D^*$-quasi-normality holds at $x^*$,
then there exists $\rho_0>0$ such that for any $\rho\geq\rho_0$, $x^*$ is also a local minimizer of the exact penalization problem
\begin{eqnarray*}
\min_{x\in \Omega}  && f(x)+\Psi(x)+ \rho \left(\|g(x)_+\| + \|h(x)\|\right).
\end{eqnarray*}
\end{theorem}
\begin{proof}
By the local optimality of $x^*$, it is easy to see that $(x^*,y^*)$ with $y^*:=\Psi(x^*)$ is a local minimizer of the following auxiliary problem:
\begin{eqnarray*}
\begin{array}{rl}
\min        & f(x)+y\\[4pt]
{\rm s.t.} & (x,y)\in \Lambda,
\end{array}
\end{eqnarray*}
where $\Lambda$ is defined in Lemma \ref{lema nerror}. Denote by $\ell$  the Lipschitz constant of the objective function $f(x)+y$ around $(x^*,y^*)$. By Clarke's exact penalization principle \cite[Proposition 2.4.3]{c},  there exists $\delta_1>0$ such that
\begin{equation*}
f(x^*) + y^* \leq f(x) +y + \ell {\rm dist}_\Lambda(x,y)\quad \forall (x,y)\in \cB_\delta(x^*,y^*).
\end{equation*}
Then it follows from Lemma \ref{lema nerror} that there exist $\delta_2\in(0,\delta_1)$ and $\kappa>0$ such that for all $(x,y)\in \cB_{\delta_2}(x^*,y^*)\cap (\Omega\times \Re)$,
\begin{eqnarray}\label{thm exact11}
f(x^*) + y^* &\leq& f(x) +y + \ell {\rm dist}_\Lambda(x,y)\nonumber\\
             &\leq& f(x) +y + \kappa \ell\left(|\Psi(x)-y|+ \|g(x)_+\| + \|h(x)\|\right).
\end{eqnarray}
By the continuity of $\Psi$, we may choose $\delta\in (0,\delta_2)$ such that $(x,\Psi(x))\in \cB_{\delta_2}(x^*,y^*)$ for any $x\in \cB_\delta(x^*)$. Then the desired result follows immediately from \eqref{thm exact11} by letting $y=\Psi(x)$ and $\rho_0:=\kappa \ell$. \qquad $\square$
\end{proof}

\begin{acknowledgements}
We thank the referees for their helpful suggestions and comments that have helped us to  improve the presentation of the paper. We would also like to thank Jim Burke for a discussion on the topic of this research.
\end{acknowledgements}



\end{document}